\documentclass[pagesize,pdftex]{scrartcl}
\usepackage{latexsym} 
\usepackage{amsmath}
\usepackage{amsthm}
\usepackage{amsfonts}
\usepackage{amssymb}
\usepackage{amsxtra}
\usepackage{amscd}
\usepackage{ifthen}
\usepackage{graphicx}
\usepackage{enumerate}
\usepackage{mathrsfs}
\usepackage[pagebackref=true]{hyperref}

\hypersetup{
pdfauthor={Mads Kyed},
pdftitle={Maximal regularity of the time-periodic Navier-Stokes system},
breaklinks=true,
colorlinks=true,
linkcolor=blue,
citecolor=blue,
urlcolor=blue,
filecolor=blue,
}

\numberwithin{equation}{section} 
\setkomafont{title}{\normalfont}

\newenvironment{pdeq}{ \left\{ \begin{aligned}}{\end{aligned}\right.}

\newcommand{\np}[1]{(#1)}
\newcommand{\bp}[1]{\big(#1\big)}
\newcommand{\bb}[1]{\big[#1\big]}
\newcommand{\Bp}[1]{\bigg(#1\bigg)}
\newcommand{\Bb}[1]{\bigg[#1\bigg]}
\newcommand{\cali}{{\mathcal I}}

\newcommand{\calp}{{\mathcal P}}

\newcommand{\calt}{{\mathcal T}}

\newcommand{\calx}{{\mathcal X}}
\newcommand{\caly}{{\mathcal Y}}

\newcommand{\R}{\mathbb{R}}
\newcommand{\Z}{\mathbb{Z}}
\newcommand{\CNumbers}{\mathbb{C}}
\newcommand{\N}{\mathbb{N}}
\DeclareMathOperator{\e}{e}

\DeclareMathOperator{\id}{Id}
\DeclareMathOperator{\Div}{div}

\DeclareMathOperator{\supp}{supp}

\newcommand{\ra}{\rightarrow}

\newcommand{\set}[1]{\ensuremath{\{#1\}}}

\newcommand{\setc}[2]{\ensuremath{\{#1\ \lvert\ #2\}}}
\newcommand{\setcl}[2]{\ensuremath{\bigl\{#1\ \lvert\ #2\bigr\}}}

\newcommand{\closure}[2]{\overline{#1}^{#2}}
\newcommand{\proj}{\calp}
\newcommand{\projsymbol}{\kappa_0}
\newcommand{\projcompl}{\calp_\bot}
\newcommand{\hproj}{\calp_H}
\newcommand{\quotientmap}{\pi}
\newcommand{\grp}{G}
\newcommand{\dualgrp}{\widehat{G}}

\newcommand{\grpH}{H}
\newcommand{\dualgrpH}{\widehat{H}}

\newcommand{\idmatrix}{I}
\newcommand{\Rnper}{\R^n\times(0,\per)}

\newcommand{\grad}{\nabla}
\newcommand{\gradmap}{\mathrm{grad}}
\newcommand{\gradmapinverse}{\mathrm{grad}^{-1}}

\newcommand{\dx}{{\mathrm d}x}

\newcommand{\dg}{{\mathrm d}g}
\newcommand{\ds}{{\mathrm d}s}
\newcommand{\dt}{{\mathrm d}t}

\newcommand{\dxi}{{\mathrm d}\xi}

\newcommand{\SR}{\mathscr{S}}

\newcommand{\TDR}{\mathscr{S^\prime}}
\newcommand{\DDR}{\mathcal{D^\prime}}

\newcommand{\linf}[2]{\langle #1, #2\rangle}

\newcommand{\ft}[1]{\widehat{#1}}
\newcommand{\ift}[1]{#1^\vee}
\newcommand{\FT}{\mathscr{F}}
\newcommand{\iFT}{\mathscr{F}^{-1}}

\newcommand{\norm}[1]{\lVert#1\rVert}
\newcommand{\oseennorm}[2]{\norm{#1}_{#2,\rey,\mathrm{Oseen}}}
\newcommand{\oseennormtwod}[2]{\norm{#1}_{#2,\rey,\mathrm{Oseen2D}}}
\newcommand{\stokesnorm}[2]{\norm{#1}_{#2,\mathrm{Stokes}}}

\newcommand{\snorm}[1]{{\lvert #1 \rvert}}
\newcommand{\snorml}[1]{{\bigl\lvert #1 \big\rvert}}

\newcommand{\WSR}[2]{W^{#1,#2}}

\newcommand{\CR}[1]{C^{#1}}  
\newcommand{\LR}[1]{L^{#1}}

\newcommand{\LRloc}[1]{L^{#1}_{loc}} 
\newcommand{\CRi}{\CR \infty}
\newcommand{\CRci}{\CR \infty_0}
\newcommand{\CRc}[1]{\CR{#1}_0}
\newcommand{\LRsigma}[1]{L^{#1}_{\sigma}} 
\newcommand{\gradspace}[1]{\mathscr{G}^{#1}}

\newcommand{\WSRsigma}[2]{W^{#1,#2}_{\sigma}} 
\newcommand{\WSRsigmacompl}[2]{W^{#1,#2}_{\sigma,\bot}} 
\newcommand{\LRsigmacompl}[1]{L^{#1}_{\sigma,\bot}}

\newcommand{\CRcisigma}{\CR{\infty}_{0,\sigma}}

\newcommand{\CRper}{\CR{\infty}_{\mathrm{per}}}

\newcommand{\CRciper}{\CR{\infty}_{0,\mathrm{per}}}
\newcommand{\CRcisigmaper}{\CR{\infty}_{0,\sigma,\mathrm{per}}}
\newcommand{\WSRsigmaper}[1]{W^{#1}_{\sigma,\mathrm{per}}} 
\newcommand{\WSRsigmapercompl}[2]{W^{#1,#2}_{\sigma,\mathrm{per},\bot}} 
\newcommand{\xoseen}[1]{{X}^{#1}_{\sigma,\mathrm{Oseen}}}
\newcommand{\xoseentwod}[1]{{X}^{#1}_{\sigma,\mathrm{Oseen2D}}}
\newcommand{\xpres}[1]{{X}^{#1}_{\mathrm{pres}}}

\newcommand{\xstokes}[1]{{X}^{#1}_{\sigma,\mathrm{Stokes}}}
\newcommand{\maxregspacelinns}[1]{X^{#1}}
\newcommand{\maxregspacelinnss}[1]{X_{\mathrm{s}}^{#1}}
\newcommand{\maxregspacelinnstp}[1]{X_{\mathrm{tp}}^{#1}}

\newcommand{\vvel}{v}

\newcommand{\wvel}{w}

\newcommand{\twvel}{\tilde{w}}

\newcommand{\uvel}{u}
\newcommand{\uvelft}{\ft{u}}
\newcommand{\uvels}{\uvel_{\mathrm{s}}}
\newcommand{\uveltp}{\uvel_{\mathrm{tp}}}
\newcommand{\upress}{\upres_{\mathrm{s}}}
\newcommand{\uprestp}{\upres_{\mathrm{tp}}}

\newcommand{\upres}{\mathfrak{p}}

\newcommand{\upresft}{\ft{\mathfrak{p}}}

\newcommand{\tuvel}{\tilde{u}}
\newcommand{\tupres}{\tilde{\mathfrak{p}}}
\newcommand{\ouvel}{\overline{u}}
\newcommand{\oupres}{\overline{\mathfrak{p}}}
\newcommand{\tf}{\tilde{f}}

\newcommand{\ALgeneric}{\mathrm{A}}
\newcommand{\ALTP}{{\ALgeneric}_{\mathrm{TP}}}
\newcommand{\ALTPinverse}{{\ALgeneric}^{-1}_{\rm TP}}

\newcommand{\tin}{\text{in }}
\newcommand{\tif}{\text{if }}

\newcommand{\tand}{\text{and }}

\newcommand{\half}{\frac{1}{2}}

\renewcommand{\epsilon}{\varepsilon}
\renewcommand{\phi}{\varphi}
\newcommand{\rey}{\lambda}

\newcommand{\tay}{\calt}
\newcommand{\per}{\tay}
\newcommand{\iper}{\frac{1}{\tay}}
\newcommand{\perf}{\frac{2\pi}{\tay}}
\newcommand{\iperf}{\frac{\tay}{2\pi}}

\newcommand{\dualrho}{\hat{\rho}}
\newcommand{\mmultiplier}{m}
\newcommand{\Mmultiplier}{M}

\newcommand{\polynomial}{P}

\newcommand{\fs}{f_{\mathrm{s}}}
\newcommand{\ftp}{f_{\mathrm{tp}}}
\newcommand{\bijection}{\Pi}
\newcommand{\bijectioninv}{\Pi^{-1}}
\newcommand{\newCCtr}[2][d]{
\newcounter{#2}\setcounter{#2}{0}
\expandafter\xdef\csname kyedtheconst#2\endcsname{#1}
}
\newcommand{\Cc}[2][nolabel]{
\stepcounter{#2}
\expandafter\ensuremath{\csname kyedtheconst#2\endcsname_{\arabic{#2}}}
\ifthenelse{\equal{#1}{nolabel}}
{}
{\expandafter\xdef\csname kyedconst#1\endcsname
{\expandafter\ensuremath{\csname kyedtheconst#2\endcsname_{\arabic{#2}}}}}
}
\newcommand{\CcSetCtr}[2]{
\setcounter{#1}{#2}
}
\newcommand{\Cclast}[1]{
\expandafter\ensuremath{\csname kyedtheconst#1\endcsname_{\arabic{#1}}}
}
\newcommand{\Ccllast}[1]{
\addtocounter{#1}{-1}
\expandafter\ensuremath{\csname kyedtheconst#1\endcsname_{\arabic{#1}}}
\addtocounter{#1}{1}
}
\newcommand{\const}[1]{
\expandafter{\ifcsname kyedconst#1\endcsname
  \csname kyedconst#1\endcsname
\else
  \errmessage{Undefined Kyedconstant #1.}%
\fi}
}

\theoremstyle{plain}
\newtheorem{thm}{Theorem}[section]
\newtheorem{defn}[thm]{Definition}
\newtheorem{lem}[thm]{Lemma}

\theoremstyle{remark}
\newtheorem{rem}[thm]{Remark}

\begin{document}
\title{Maximal regularity of the time-periodic Navier-Stokes system}

\author{
Mads Kyed\\ 
Institut f\"ur Mathematik \\
Universit\"at Kassel, Germany \\
Email: {\tt mkyed@mathematik.uni-kassel.de}
}

\date{August 26, 2013}
\maketitle

\begin{abstract}
Time-periodic solutions to the linearized Navier-Stokes system in the $n$-dimensional whole-space are investigated. For time-periodic data in $\LR{q}$-spaces,
maximal regularity and corresponding a priori estimates for the associated time-periodic solutions are established.
More specifically, a Banach space of time-periodic vector fields is identified with the property that the linearized Navier-Stokes operator maps this space homeomorphically onto the $\LR{q}$-space 
of time-periodic data. 
\end{abstract}


\newCCtr[C]{C}
\newCCtr[M]{M}
\newCCtr[B]{B}
\newCCtr[\epsilon]{eps}
\CcSetCtr{eps}{-1}

\section{Introduction}

We investigate the time-periodic, linearized Navier-Stokes system in the $n$-dimensional whole-space
with $n\geq 2$. More specifically, we consider the linearized Navier-Stokes system
\begin{align}\label{intro_linnssystem}
\begin{pdeq}
&\partial_t\uvel -\Delta\uvel -\rey\partial_1\uvel + \grad\upres = f && \tin\R^n\times\R,\\
&\Div\uvel =0 && \tin\R^n\times\R
\end{pdeq}
\end{align} 
for an Eulerian velocity field $\uvel:\R^n\times\R\ra\R^n$ and pressure term
$\upres:\R^n\times\R\ra\R$ as well as data $f:\R^n\times\R\ra\R^n$ that are all $\per$-time-periodic, that is, 
\begin{align}\label{intro_timeperiodicsolution}
\begin{aligned}
&\forall (x,t)\in\R^n\times\R:\quad \uvel(x,t) = \uvel(x,t+\per)\quad\tand\quad \upres(x,t) = \upres(x,t+\per)
\end{aligned}
\end{align}
and 
\begin{align}\label{intro_timeperiodicdata}
\begin{aligned}
&\forall (x,t)\in\R^n\times\R:\quad f(x,t) = f(x,t+\per).
\end{aligned}
\end{align}
The time period $\per>0$ remains fixed.  The constant $\rey\in\R$ determines the type of linearization. If $\rey=0$ the system 
\eqref{intro_linnssystem} is a Stokes system, while $\rey\neq 0$ turns \eqref{intro_linnssystem} into an Oseen system.
These are the two possible ways to linearize the Navier-Stokes system.             

The main goal in this paper is to identify a Banach space $\maxregspacelinns{q}$ with the property that  
for any vector field $f\in\LR{q}\bp{\Rnper}^n$ satisfying \eqref{intro_timeperiodicdata} there is a unique solution $(\uvel,\upres)$ in $\maxregspacelinns{q}$ to \eqref{intro_linnssystem}--\eqref{intro_timeperiodicsolution} such that
\begin{align}\label{intro_aprioriest}
\norm{(\uvel,\upres)}_{\maxregspacelinns{q}} \leq C\,\norm{f}_q,
\end{align}
with a constant $C$ depending only on $\rey,\per,q$ and $n$.
Note that any $f\in\LR{q}\bp{\R^n\times(0,\per)}^n$ has a trivial extension to a time-periodic 
vector field satisfying \eqref{intro_timeperiodicdata}.
More precisely, we wish to establish a functional analytic setting in which \eqref{intro_linnssystem}--\eqref{intro_timeperiodicdata} can be written on an 
operator form
\begin{align*}
\ALgeneric(\uvel,\upres) = f
\end{align*}  
such that the operator
\begin{align*}
\ALgeneric:\maxregspacelinns{q}\ra\LR{q}\bp{\Rnper}^n
\end{align*}
is a homeomorphism. We say that a function space $\maxregspacelinns{q}$ with this property establishes maximal regularity in the $\LR{q}$-setting for the linearized, time-periodic Navier-Stokes system. In this context, regularity refers not only to the order of differentiability of 
$(\uvel,\upres)$, but also, and in particular, to the summability of $\uvel$ and $\upres$. 

In order to identify the space $\maxregspacelinns{q}$, we shall employ the theory of Fourier multipliers. This may not seem surprising, as similar
results for both the corresponding steady-state and initial-value problem are traditionally established using Fourier multipliers. However, it is 
not directly clear how to employ the Fourier transform on the space-time domain $\Rnper$. The main idea behind the approach presented in the 
following is to identify $\Rnper$ with the group $\grp:=\R^n\times\R/\per\Z$. Equipped with the canonical topology, $\grp$ is a locally compact 
abelian group. As such, there is a naturally defined Fourier transform $\FT_\grp$ associated to $\grp$. 
Moreover, as $\grp$ inherits a differentiable structure from $\R^n\times\R$ in a canonical way, we may view \eqref{intro_linnssystem} as a system of differential equations on $\grp$. 
Employing the Fourier transform $\FT_\grp$, we then obtain a representation of the solution in terms of a Fourier multiplier defined
on the dual group $\dualgrp$. Based on this representation, the space $\maxregspacelinns{q}$ and corresponding \textit{a priori} estimate 
\eqref{intro_aprioriest} will be established. 
Since multiplier theorems like the theorems of Mihlin, Lizorkin or Marcinkiewicz are 
only available in an Euclidean setting $\R^m$, and \emph{not} in the general setting of group multipliers, we shall employ a so-called transference principle.
More specifically, we shall use a theorem, which in its original form is due to \textsc{de Leeuw} \cite{Leeuw1965}, that enables us to 
study the properties of a multiplier defined on $\dualgrp$ in terms of a corresponding multiplier defined on $\R^{n+1}$. 

\section{Statement of main result}

In order to state the main result, we first introduce an appropriate decomposition of the problem. For this purpose, we express the data
$f$ as a sum of a time-independent part $\fs$ and a time-periodic part $\ftp$ with vanishing time-average over the periodic, that is, 
$f=\fs+\ftp$ with $\fs:\R^n\ra\R^n$
and $\ftp:\R^n\times\R\ra\R^n$ satisfying \eqref{intro_timeperiodicdata} and $\int_0^\per \ftp(x,t)\,\dt=0$. 
Based on this decomposition of $f$, we 
split the investigation of \eqref{intro_linnssystem} into two parts: A steady-state problem with data $\fs$ 
corresponding to time-independent unknowns $(\uvels,\upress)$, and a time-periodic problem with data $\ftp$ corresponding to
time-periodic unknowns $(\uveltp,\uprestp)$ with vanishing time-average over the periodic.
In this way, we obtain a solution $(\uvel,\upres)=(\uvels+\uveltp,\upress+\uprestp)$ to \eqref{intro_linnssystem}. We shall identify function spaces 
$\maxregspacelinnss{q}$ and 
$\maxregspacelinnstp{q}$ that establish maximal regularity for $\uvels$ and $\uveltp$ separately.
We thus obtain $\maxregspacelinns{q}=\bp{\maxregspacelinnss{q}\oplus\maxregspacelinnstp{q}}\times\xpres{q}$, 
where $\xpres{q}$ is the function space that establishes maximal regularity for the pressure term $\upres$.
It turns out that the regularity, more specifically the summability, of the velocity fields in $\maxregspacelinnss{q}$ and $\maxregspacelinnstp{q}$ are substantially different. 

We denote points in $\R^n\times\R$ by $(x,t)$, and refer to $x$ as the spatial and $t$ as the time variable. Our main result is formulated in terms of the function spaces:  
\begin{align*}
&\CRper(\R^n\times\R) := \setcl{U\in\CRi(\R^n\times\R)}{\forall t\in\R:\ U(\cdot,t+\per)=U(\cdot,t)},\\
&\CRciper(\R^n\times[0,\per]) := \setcl{u\in\CRci(\R^n\times[0,\per])}{\exists U\in\CRper(\R^n\times\R):\ u=U_{|\R^n\times[0,\per]}},\\
&\CRcisigmaper(\R^n\times[0,\per]) := \setcl{u\in\CRciper(\R^n\times[0,\per])^n}{\Div_x u = 0},
\end{align*}
for $q\in(1,\infty)$ the Lebesgue space of solenoidal vector fields
\begin{align*}
&\LRsigma{q}\bp{\R^n\times(0,\per)}:=\closure{\CRcisigmaper(\R^n\times[0,\per])}{\norm{\cdot}_{q}}=
\setc{u\in\LR{q}\bp{\R^n\times(0,\per)}}{\Div_x u =0},\\
&\norm{u}_{q} := \norm{u}_{\LR{q}\bp{\R^n\times(0,\per)}} 
\end{align*}
and the Sobolev space of time-periodic, solenoidal, vector fields  
\begin{align*}
\begin{aligned}
&\WSRsigmaper{2,1,q}\bp{\R^n\times(0,\per)}:= \closure{\CRcisigmaper(\R^n\times[0,\per])}{\norm{\cdot}_{2,1,q}},\\
&\norm{u}_{2,1,q} := 
\Bp{\sum_{\snorm{\alpha}\leq 2,\snorm{\beta}\leq 1} \norm{\partial_x^\alpha u}^q_{\LR{q}\bp{\R^n\times(0,\per)}} + 
\norm{\partial_t^\beta u}^q_{\LR{q}\bp{\R^n\times(0,\per)}}}^{1/q}. 
\end{aligned}
\end{align*}
In order to incorporate the decomposition described above on the level of function spaces, we introduce the projection
\begin{align}\label{intro_defofproj} 
&\proj:\CRciper(\R^n\times[0,\per])\ra\CRciper(\R^n\times[0,\per]),\quad \proj u(x,t):=\iper\int_0^\per u(x,s)\,\ds
\end{align}
and put $\projcompl:=\id-\proj$. Observe that $\proj$ and $\projcompl$ decompose a time-periodic vector field $u$ into a time-independent part
$\proj u$ and a time-periodic part $\projcompl u$ with vanishing time-average over the period. Both projections extend by continuity to
bounded operators on $\LRsigma{q}\bp{\R^n\times(0,\per)}$ and $\WSRsigmaper{2,1,q}\bp{\R^n\times(0,\per)}$.
Finally, we put
\begin{align*}
&\WSRsigmapercompl{2,1}{q}\bp{\R^n\times(0,\per)}:= 
\projcompl\WSRsigmaper{2,1,q}\bp{\R^n\times(0,\per)},\\
&\LRsigmacompl{q}\bp{\R^n\times(0,\per)} := \projcompl\LRsigma{q}\bp{\R^n\times(0,\per)}.
\end{align*}

Our first main theorem states that $\maxregspacelinnstp{q}=\WSRsigmapercompl{2,1}{q}\bp{\R^n\times(0,\per)}$. 
We state the theorem in a setting of solenoidal vector-fields and thereby eliminate the need to characterize the pressure term. 
More specifically, we show:

\begin{thm}\label{mainresult_tpmaxregthm}
Let $q\in(1,\infty)$. For each $f\in\LRsigmacompl{q}\bp{\R^n\times(0,\per)}$ there is a unique 
solution $\uvel\in\WSRsigmapercompl{2,1}{q}\bp{\R^n\times(0,\per)}$ to \eqref{intro_linnssystem}, that is, a solution in the sense that the trivial 
extensions of $\uvel$ and $f$ to time-periodic vector fields 
on $\R^n\times\R$ together with $\upres=0$ constitutes a solution in the standard sense of distributions. Moreover,
\begin{align}\label{mainresult_tpmaxregthmEstimate}
\norm{\uvel}_{2,1,q}\leq \Cc[MainThmConst]{C}\,\norm{f}_q,
\end{align}
where $\Cclast{C}=\Cclast{C}(\rey,\per,n,q)$.
\end{thm}

The other maximal regularity space 
$\maxregspacelinnss{q}$ is already well-known  from the theory of steady-state, linearized Navier-Stokes systems. In order to characterize it, one
has to distinguish between $\rey=0$ (Stokes case) and $\rey\neq 0$ (Oseen case). Moreover, the case $n=2$ has to be treated separately.
For $n\geq 3$, $\rey=0$, $q\in(1,\frac{n}{2})$ we put
\begin{align*}
&\xstokes{q}(\R^n):=\setcl{\vvel\in\LRloc{1}(\R^n)^n}{\Div\vvel=0,\ \stokesnorm{\vvel}{q}<\infty},\\
&\stokesnorm{\vvel}{q} := \norm{\vvel}_{\frac{nq}{n-2q}} + \norm{\grad\vvel}_{\frac{nq}{n-q}} + \norm{\grad^2\vvel}_q.  
\end{align*}
For $n\geq 3$, $\rey\neq 0$, $q\in(1,\frac{n+1}{2})$ we put
\begin{align*}
&\xoseen{q}(\R^n):=\setcl{\vvel\in\LRloc{1}(\R^n)^n}{\Div\vvel=0,\ \oseennorm{\vvel}{q}<\infty},\\
&\oseennorm{\vvel}{q} := \snorm{\rey}^{\frac{2}{n+1}}\norm{\vvel}_{\frac{(n+1)q}{n+1-2q}} + \snorm{\rey}^{\frac{1}{n+1}}\norm{\grad\vvel}_{\frac{(n+1)q}{n+1-q}} + \snorm{\rey}\norm{\partial_1\vvel}_q + \norm{\grad^2\vvel}_q.  
\end{align*}
For $n=2$, $\rey\neq 0$, $q\in(1,\frac{3}{2})$ we put
\begin{align*}
&\xoseentwod{q}(\R^2):=\setcl{\vvel\in\LRloc{1}(\R^2)^2}{\Div\vvel=0,\ \oseennormtwod{\vvel}{q}<\infty},\\
&\oseennormtwod{\vvel}{q} := \snorm{\rey}^{\frac{2}{3}}\norm{\vvel}_{\frac{3q}{3-2q}} + \snorm{\rey}^{\frac{1}{3}}\norm{\grad\vvel}_{\frac{3q}{3-q}} + \snorm{\rey}\norm{\partial_1\vvel}_q + \norm{\grad^2\vvel}_q \\
&\qquad\qquad\qquad\quad + \snorm{\rey}\norm{\grad\vvel_2}_q + \snorm{\rey}\norm{\vvel_2}_{\frac{2q}{2-q}}.
\end{align*}
We shall not treat the case $n=2$, $\rey=0$ explicitly; see however remark \ref{remarkOn2Dcase} below. 
To characterize the maximal regularity for the pressure $\upres$, we put
\begin{align*}
\begin{aligned}
&\xpres{q}\bp{\R^n\times(0,\per)}:=\setc{\upres\in\LRloc{1}\bp{\R^n\times(0,\per)}}{\norm{\upres}_{\xpres{q}}<\infty},\\
&\norm{\upres}_{\xpres{q}}:=\Bp{\iper\int_0^\per \norm{\upres(\cdot,t)}_{\frac{nq}{n-q}}^q + \norm{\grad_x\upres(\cdot,t)}_q^q\,\dt }^{1/q}.
\end{aligned}
\end{align*}

We are now in a position to state the theorem that establishes maximal regularity in the general $\LR{q}$-setting. 
\begin{thm}[Maximal regularity in $\LR{q}$-setting]\label{mainresult_maxregthm}
Let
\begin{align}
&\maxregspacelinnss{q}(\R^n) := \xstokes{q}(\R^n) &&\text{if\ } n\geq 3,\ \rey=0,\ q\in\bp{1,\frac{n}{2}},\label{mainresult_maxregthmDefOfXs1}\\
&\maxregspacelinnss{q}(\R^n) := \xoseen{q}(\R^n) &&\text{if\ } n\geq 3,\ \rey\neq 0,\ q\in\bp{1,\frac{n+1}{2}},\label{mainresult_maxregthmDefOfXs2}\\
&\maxregspacelinnss{q}(\R^n) := \xoseentwod{q}(\R^2) &&\text{if\ } n= 2,\ \rey\neq 0,\ q\in\bp{1,\frac{3}{2}}.\label{mainresult_maxregthmDefOfXs3}
\end{align} 
For every $f\in\LR{q}\bp{\R^n\times(0,\per)}^n$ there is unique solution\footnote{We use $\oplus$ to denote either the sum of two subspaces 
of $\LRloc{1}\bp{\R^n\times(0,\per)}$ whose intersection contains only $0$, or the sum of two functions from such subspaces.}
\begin{align*}
(\uvel,\upres)=(\vvel\oplus\wvel,\upres)\in \Bp{\maxregspacelinnss{q}(\R^n)\oplus\WSRsigmapercompl{2,1}{q}
\bp{\R^n\times(0,\per)}}\times\xpres{q}\bp{\R^n\times(0,\per)} 
\end{align*}
to \eqref{intro_linnssystem} in the sense that the trivial extensions of $(\uvel,\upres)$ and $f$ to time-periodic vector fields 
on $\R^n\times\R$ that satisfy \eqref{intro_timeperiodicsolution}--\eqref{intro_timeperiodicdata} constitute a solution in the standard sense of distributions.
Moreover, 
\begin{align}\label{mainresult_maxregthmEst}
\norm{\vvel}_{\maxregspacelinnss{q}} + \norm{\wvel}_{2,1,q} + \norm{\upres}_{\xpres{q}} \leq \Cc[MainThmConst]{C}\,\norm{f}_q,
\end{align}
where $\Cclast{C}=\Cclast{C}(\rey,\per,n,q)$
\end{thm}

\begin{rem}\label{remarkOn2Dcase}
The projection $\proj$ induces the decomposition 
\begin{align*}
\LRsigma{q}\bp{\R^n\times(0,\per)}=\LRsigma{q}(\R^n)\oplus\LRsigmacompl{q}\bp{\R^n\times(0,\per)}.
\end{align*}
Theorem \ref{mainresult_tpmaxregthm} states that 
\begin{align*}
\partial_t -\Delta -\rey\partial_1:\ \WSRsigmapercompl{2,1}{q}\bp{\R^n\times(0,\per)} \ra \LRsigmacompl{q}\bp{\R^n\times(0,\per)}
\end{align*}
is a homeomorphism. Theorem \ref{mainresult_maxregthm} is therefore basically a consequence of Theorem \ref{mainresult_tpmaxregthm} combined with the well-known fact that
in the steady-state stetting the same operator maps  $\maxregspacelinnss{q}(\R^n)$, defined in \eqref{mainresult_maxregthmDefOfXs1}--\eqref{mainresult_maxregthmDefOfXs3}, homeomorphically onto $\LRsigma{q}(\R^n)$.
Generally, if one can identify steady-state function spaces $\calx_\sigma(\R^n)$ and $\caly_\sigma(\R^n)$ of solenoidal vector fields 
such that 
$-\Delta -\rey\partial_1:\ \calx_\sigma(\R^n)\ra\caly_\sigma(\R^n)$
is a homeomorphism, Theorem \ref{mainresult_tpmaxregthm} can be employed to show that
\begin{align*}
\partial_t -\Delta -\rey\partial_1:\ \calx_\sigma(\R^n)\oplus\WSRsigmapercompl{2,1}{q}\bp{\R^n\times(0,\per)} \ra \caly_\sigma(\R^n)\oplus\LRsigmacompl{q}\bp{\R^n\times(0,\per)}
\end{align*}
is a homeomorphism. If one for example wishes to investigate a case in which the parameters $n,\rey,q$ are not covered by
\eqref{mainresult_maxregthmDefOfXs1}--\eqref{mainresult_maxregthmDefOfXs3}, one must simply identify such function spaces $\calx_\sigma(\R^n)$ and $\caly_\sigma(\R^n)$ for this particular choice of parameters. In Theorem \ref{mainresult_tpmaxregthm} we have covered only the cases where 
these spaces are well-known by what can be considered standard theory.
\end{rem}

\section{Notation}

Points in $\R^n\times\R$ are denoted by $(x,t)$ with $x\in\R^n$ and $t\in\R$.
We refer to $x$ as the spatial and to $t$ as the time variable. 

For a sufficiently regular function $u:\R^n\times\R\ra\R$, we put $\partial_i u:=\partial_{x_i} u$.
For any multiindex $\alpha\in\N_0^n$, we let $\partial_x^\alpha\uvel:= \sum_{j=1}^n \partial_j^{\alpha_j}\uvel$
and put $\snorm{\alpha}:=\sum_{j=1}^n \alpha_j$. Moreover, for $x\in\R^n$ we let 
$x^\alpha:=x_1^{\alpha_1}\cdots x_n^{\alpha_n}$.
Differential operators act only in the spatial variable unless otherwise indicated. For example, 
we denote by $\Delta u$ the Laplacian of $u$ with respect to the spatial variable, that is, 
$\Delta u:=\sum_{j=1}^n\partial_j^2 u$. For a vector field $u:\R^n\times\R\ra\R^n$ we 
let $\Div u:=\sum_{j=1}^n\partial_j u_j$ denote the divergence of $u$.

For two vectors $a,b\in\R^n$ we let $a \otimes b\in\R^{n\times n}$ denote the tensor 
with $(a\otimes b)_{ij}:=a_ib_j$. 
We denote by $\idmatrix$ the identity tensor 
$\idmatrix\in\R^{n\times n}$.

For a vector space $X$ and $A,B\subset X$, we write $X=A\oplus B$ iff $A$ and $B$ are subspaces of $X$ with 
$A\cap B=\set{0}$ and $X=A+B$. We also write $a\oplus b$ for elements of $A\oplus B$.

Constants in capital letters in the proofs and theorems are global, while constants in small letters are local to the proof in which they appear.
   
\section{Reformulation in a group setting}

In the following, we let $\grp$ denote the group
\begin{align}\label{lt_defofgrp}
\grp:=\R^n\times\R/\per\Z
\end{align} 
with addition as the group operation.
We shall reformulate \eqref{intro_linnssystem}--\eqref{intro_timeperiodicdata} and the main theorems in a setting of functions defined on $\grp$. For this purpose, we must first introduce a topology and an appropriate differentiable structure on $\grp$. It is then possible to define Lebesgue and Sobolev spaces on $\grp$.

\subsection{Differentiable structure, distributions and Fourier transform}\label{lt_differentiablestructuresubsection}

The topology and differentiable structure on $\grp$ is inherited from $\R^n\times\R$. More precisely, 
we equip $\grp$ with the quotient topology induced by the canonical quotient mapping
\begin{align}\label{lt_quotientmap}
\quotientmap :\R^n\times\R \ra \R^n\times\R/\per\Z,\quad \quotientmap(x,t):=(x,[t]).
\end{align}
Equipped with the quotient topology, $\grp$ becomes a locally compact abelian group. 
We shall use the restriction 
\begin{align*}
\bijection:\R^n\times[0,\per)\ra\grp,\quad \bijection:=\pi_{|\R^n\times[0,\per)}
\end{align*}
to identify $\grp$ with the domain $\R^n\times[0,\per)$, as $\bijection$ is clearly a (continuous) bijection. 

Via $\bijection$, one can identify the Haar measure $\dg$ on $\grp$ as the product of the Lebesgue measure on $\R^n$ and the Lebesgue measure on $[0,\per)$.
The Haar measure is unique up-to a normalization factor, which we choose such that
\begin{align*}
\forall\uvel\in\CRc{}(\grp):\quad \int_\grp \uvel(g)\,\dg = \iper\int_0^\per\int_{\R^n} \uvel\circ\bijection(x,t)\,\dx\dt,
\end{align*}
where $\CRc{}(\grp)$ denotes the space of continuous functions of compact support.  
For the sake of convenience, we will omit the $\bijection$ in integrals of $\grp$-defined functions with respect to $\dx\dt$, that is, instead of 
$\iper\int_0^\per\int_{\R^n} \uvel\circ\bijection(x,t)\,\dx\dt$ we simply write $\iper\int_0^\per\int_{\R^n} \uvel(x,t)\,\dx\dt$.

Next, we define by
\begin{align}\label{lt_smoothfunctionsongrp}
\CRi(\grp):=\setc{\uvel:\grp\ra\R}{\uvel\circ\quotientmap \in\CRi(\R^n\times\R)}
\end{align}
the space of smooth functions on $\grp$. For $\uvel\in\CRi(\grp)$ we define derivatives 
\begin{align}\label{lt_defofgrpderivatives}
\forall(\alpha,\beta)\in\N_0^n\times\N_0:\quad \partial_t^\beta\partial_x^\alpha\uvel := \bb{\partial_t^\beta\partial_x^\alpha (\uvel\circ\quotientmap)}\circ\bijectioninv.
\end{align}
It is easy to verify for $\uvel\in\CRi(\grp)$ that also $\partial_t^\beta\partial_x^\alpha\uvel\in\CRi(\grp)$. We further introduce the subspace 
\begin{align*}
\CRci(\grp):=\setc{\uvel\in\CRi(\grp)}{\supp\uvel\text{ is compact}}
\end{align*}
of compactly supported smooth functions. Clearly, $\CRci(\grp)\subset\CRc{}(\grp)$.

With a differentiable structure defined on $\grp$ via \eqref{lt_smoothfunctionsongrp}, we can introduce the space of tempered distributions on $\grp$.
For this purpose, we first recall the Schwartz-Bruhat space of generalized Schwartz functions; see for example \cite{Bruhat61}. More precisely, we define for $\uvel\in\CRi(\grp)$ the semi-norms
\begin{align*}
\forall(\alpha,\beta,\gamma)\in\N_0^n\times\N_0\times\N_0^n:\quad 
\rho_{\alpha,\beta,\gamma}(\uvel):=\sup_{(x,t)\in\grp} \snorm{x^\gamma\partial_t^\beta\partial_x^\alpha\uvel(x,t)},
\end{align*}
and put 
\begin{align*}
\SR(\grp):=\setc{\uvel\in\CRi(\grp)}{\forall(\alpha,\beta,\gamma)\in\N_0^n\times\N_0\times\N_0^n:\ \rho_{\alpha,\beta,\gamma}(\uvel)<\infty}.
\end{align*}
Clearly, $\SR(\grp)$ is a vector space, and $\rho_{\alpha,\beta,\gamma}$ a semi-norm on $\SR(\grp)$. We endow $\SR(\grp)$ with the semi-norm 
topology induced by the family $\setcl{\rho_{\alpha,\beta,\gamma}}{(\alpha,\beta,\gamma)\in\N_0^n\times\N_0\times\N_0^n}$.
The topological dual space $\TDR(\grp)$ of $\SR(\grp)$ is then well-defined. We equip $\TDR(\grp)$ with the weak* topology and refer to it 
as the space of tempered distributions on $\grp$. Observe that both $\SR(\grp)$ and $\TDR(\grp)$ remain closed under multiplication 
by smooth functions that have at most polynomial growth with respect to the spatial variables.

For a tempered distribution $\uvel\in\TDR(\grp)$, distributional derivatives 
$\partial_t^\beta\partial_x^\alpha\uvel\in\TDR(\grp)$ are defined by duality in the usual manner:
\begin{align*}
\forall \psi\in\SR(\grp):\ \linf{\partial_t^\beta\partial_x^\alpha\uvel}{\psi}:=\linf{\uvel}{(-1)^{\snorm{(\alpha,\beta)}}\partial_t^\beta\partial_x^\alpha\psi}. 
\end{align*}
It is easy to verify that $\partial_t^\beta\partial_x^\alpha\uvel$ is well-defined as an element in $\TDR(\grp)$.
For tempered distributions on $\grp$, we keep the convention that differential operators like $\Delta$ and $\Div$ act only in the 
spatial variable $x$ unless otherwise indicated. 

We shall also introduce tempered distributions on $\grp$'s dual group $\dualgrp$. We associate each $(\xi,k)\in\R^n\times\Z$ with 
the character 
$\chi:\grp\ra\CNumbers,\ \chi(x,t):=\e^{ix\cdot\xi+ik\perf t}$
on $\grp$. It is standard to verify that all characters are of this form, and we can thus identify
$\dualgrp = \R^n\times\Z$. By default, $\dualgrp$ is equipped with the compact-open topology, which in this case coincides with the product of the Euclidean topology on $\R^n$ and 
the discrete topology on $\Z$. The Haar measure on $\dualgrp$ is simply the product of the Lebesgue measure on $\R^n$ and the counting measure on $\Z$. 

A differentiable structure on $\dualgrp$ is obtained by introduction of the space 
\begin{align*}
\CRi(\dualgrp):=\setc{\wvel\in\CR{}(\dualgrp)}{\forall k\in\Z:\ \wvel(\cdot,k)\in\CRi(\R^n)}.
\end{align*}
To define the generalized Schwartz-Bruhat space on the dual group $\dualgrp$, we further introduce the semi-norms
\begin{align*}
\forall (\alpha,\beta,\gamma)\in\N_0^n\times\N_0\times\N_0^n:\ \dualrho_{\alpha,\beta,\gamma}(\wvel):= 
\sup_{(\xi,k)\in\dualgrp} \snorm{k^\beta \xi^\alpha \partial_\xi^\gamma \wvel(\xi,k)}. 
\end{align*}
We then put 
\begin{align*}
\begin{aligned} 
\SR(\dualgrp)&:=\setc{\wvel\in\CRi(\dualgrp)}{\forall (\alpha,\beta,\gamma)\in\N_0^n\times\N_0\times\N_0^n:\ \dualrho_{\alpha,\beta,\gamma}(\wvel)<\infty}.
\end{aligned}
\end{align*}
We endow the vector space $\SR(\dualgrp)$ with the semi-norm topology induced by the family
of semi-norms $\setc{\dualrho_{\alpha,\beta,\gamma}}{(\alpha,\beta,\gamma)\in\N_0^n\times\N_0\times\N_0^n}$. The topological dual space  
of $\SR(\dualgrp)$ is denoted by $\TDR(\dualgrp)$. We equip $\TDR(\dualgrp)$ with the weak* topology and refer to it 
as the space of tempered distributions on $\dualgrp$.

All function spaces have so far been defined as real vector spaces of real functions. 
Clearly, we can define them analogously as complex vector spaces of complex functions. 
When a function space is used in context with the Fourier transform, which we shall 
introduce below, we consider it as a complex vector space. 

The Fourier transform $\FT_\grp$ on $\grp$ is given by
\begin{align*}
\FT_\grp:\LR{1}(\grp)\ra\CR{}(\dualgrp),\quad \FT_\grp(\uvel)(\xi,k):=\ft{\uvel}(\xi,k):=
\iper\int_0^\per\int_{\R^n} \uvel(x,t)\,\e^{-ix\cdot\xi-ik\perf t}\,\dx\dt.
\end{align*}
If no confusion can arise, we simply write $\FT$ instead of $\FT_\grp$.
The inverse Fourier transform is formally defined by 
\begin{align*}
\iFT:\LR{1}(\dualgrp)\ra\CR{}(\grp),\quad \iFT(\wvel)(x,t):=\ift{\wvel}(x,t):=
\sum_{k\in\Z}\,\int_{\R^n} \wvel(\xi,k)\,\e^{ix\cdot\xi+ik\perf t}\,\dxi.
\end{align*}
It is standard to verify that $\FT:\SR(\grp)\ra\SR(\dualgrp)$ is a homeomorphism with $\iFT$ as the actual inverse, provided the Lebesgue measure $\dxi$ is normalized appropriately. 
By duality, $\FT$ extends to a mapping $\TDR(\grp)\ra\TDR(\dualgrp)$. More precisely, we define 
\begin{align*}
\FT:\TDR(\grp)\ra\TDR(\dualgrp),\quad \forall\psi\in\SR(\dualgrp):\ \linf{\FT(\uvel)}{\psi}:=\linf{\uvel}{\FT({\psi})}.
\end{align*}
Similarly, we define
\begin{align*}
\iFT:\TDR(\dualgrp)\ra\TDR(\grp),\quad \forall\psi\in\SR(\grp):\ \linf{\iFT(\uvel)}{\psi}:=\linf{\uvel}{\iFT({\psi})}.
\end{align*}
Clearly $\FT:\TDR(\grp)\ra\TDR(\dualgrp)$ is a homeomorphism with $\iFT$ as the actual inverse.

The Fourier transform in the setting above provides us with a calculus between the differential operators on $\grp$ and the 
polynomials on $\dualgrp$. As one easily verifies, for $\uvel\in\TDR(\grp)$ and $\alpha\in\N_0^n$, $l\in\N_0$ we have
\begin{align*}
\FT\bp{\partial_t^l\partial_x^\alpha\uvel}=i^{l+\snorm{\alpha}}\,\Big(\perf\Big)^l\,k^l\,\xi^\alpha\,\FT(\uvel)
\end{align*} 
as identity in $\TDR(\dualgrp)$.

\subsection{Sobolev spaces}\label{lt_functionspacesSection}

We let $\LR{q}(\grp)$ denote the usual Lebesgue space with respect to the Haar measure $\dg$.
A standard mollifier argument shows that $\CRci(\grp)$ is a dense subset of $\LR{q}(\grp)$.
It is standard to verify that $\LR{q}(\grp)\subset\TDR(\grp)$.

We define by
\begin{align*}
\begin{aligned}
&\WSR{2,1}{q}(\grp):=\setc{\uvel\in\LR{q}(\grp)}{\norm{\uvel}_{2,1,q}<\infty},\\
&\norm{\uvel}_{2,1,q}:=\Bp{\sum_{\snorm{\alpha}\leq 2,\snorm{\beta}\leq 1} \norm{\partial_x^\alpha\uvel}^q_q
+ \norm{\partial_t^\beta\uvel}^q_q}^{1/q}.
\end{aligned}
\end{align*}
a Sobolev space. The functions $\partial_x^\alpha\uvel$ and $\partial_t^\beta\uvel$ above 
are well-defined at the outset as tempered distributions. The condition $\norm{\uvel}_{2,1,q}<\infty$ expresses that these derivatives 
belong to $\LR{q}(\grp)$. Standard mollifier arguments yield that $\CRci(\grp)$ is dense in $\WSR{2,1}{q}(\grp)$.

We further define the following subspaces of solenoidal vector fields:
\begin{align}
&\CRcisigma(\grp):=\setc{\uvel\in\CRci(\grp)^n}{\Div\uvel=0},\\
&\LRsigma{q}(\grp):=\overline{\CRcisigma(\grp)}^{\norm{\cdot}_q},\label{lt_LRsigmadef}\\
&\WSRsigma{2,1}{q}(\grp):=\overline{\CRcisigma(\grp)}^{\norm{\cdot}_{2,1,q}}.\label{lt_AnisoWSRsigmadef}  
\end{align}
We have the following characterization of the spaces above:
\begin{lem}\label{lt_densitylemma}
For any $q\in (1,\infty)$:
\begin{align}
&\LRsigma{q}(\grp)=\setc{\uvel\in\LR{q}(\grp)^n}{\Div\uvel=0},\label{lt_densitylemmaLRsigmaCharacterization}\\
&\WSRsigma{2,1}{q}(\grp)=\setc{\uvel\in\WSR{2,1}{q}(\grp)^n}{\Div\uvel=0}.\label{lt_densitylemmaAnisoWSRsigmaCharacterization}
\end{align}
\end{lem}
The above identities are well-known if the underlying domain of the function spaces is, for example, $\R^n$. A proof can be found in \cite[Chapter III.4]{galdi:book1}. 
Simple modifications to this proof (see \cite[Lemma 3.2.1]{habil}) suffice to show the identities in the case where $\R^n$ is replaced with $\grp$. 

Next, we shall define the Helmholtz projection on the Lebesgue space $\LR{q}(\grp)^n$ of vector fields defined on $\grp$. 
For this purpose we employ the Fourier transform $\FT_\grp$ and define the Helmholtz projection as a Fourier multiplier:
\begin{defn}\label{lt_HelmholtzProjDef}
For $f\in\LR{2}(\grp)^n$ we define by
\begin{align}\label{lt_HelmholtzProjDefDef}
\hproj f := \iFT_\grp\Bb{\Bp{\idmatrix - \frac{\xi\otimes\xi}{\snorm{\xi}^2}} \ft{f}}
\end{align}
the Helmholtz projection. 
\end{defn}
It is not immediately clear from the definition of the Helmholtz projection via the Fourier multiplier in \eqref{lt_HelmholtzProjDefDef}
that $\hproj f$ is a real function if $f$ is real. This, however, is a simple consequence of the fact that the multiplier in question is even.  

Since the multiplier on the right-hand side in \eqref{lt_HelmholtzProjDefDef} is bounded, it is natural to initially define $\hproj$ on 
$\LR{2}(\grp)^n$. If namely $f\in\LR{2}(\grp)^n$, by Plancherel's theorem also $\hproj f\in\LR{2}(\grp)^n$.
We state in the following lemma that $\hproj$ can be extended to $\LR{q}(\grp)^n$, and that it is a projection with the desired properties of a Helmholtz projection.
\begin{lem}\label{lt_HelmholtzProjLem}
Let $q\in(1,\infty)$. Then $\hproj$ extends uniquely to a continuous projection $\hproj:\LR{q}(\grp)^n\ra\LR{q}(\grp)^n$. Moreover,
$\hproj\LR{q}(\grp)^n=\LRsigma{q}(\grp)$. 
\end{lem}
\begin{proof}
It is easy to see that we can define on $\grp:=\R^n\times\R/\per\Z$ the partial Fourier transforms $\FT_{\R^n}:\TDR(\grp)\ra\TDR(\grp)$ and $\FT_{\R/\per\Z}:\TDR(\grp)\ra\TDR(\grp)$ in the canonical way, and that $\FT_\grp=\FT_{\R^n}\circ\FT_{\R/\per\Z}$. Consequently, 
for $f\in\LR{2}(\grp)^n\cap\LR{q}(\grp)^n$
\begin{align*}
\iFT_\grp\Bb{\Bp{\idmatrix - \frac{\xi\otimes\xi}{\snorm{\xi}^2}} \ft{f}} = 
\iFT_{\R^n} \Bb{\Bp{\idmatrix - \frac{\xi\otimes\xi}{\snorm{\xi}^2}} \FT_{\R^n}\bp{f}}.
\end{align*}
From the boundedness of the classical Helmholtz projection on $\LR{q}(\R^n)^n$, which one recognizes on the right-hand side above, it therefore follows that $\norm{\hproj f}_{q}\leq \norm{f}_q$. Thus, $\hproj$ extends uniquely to a continuous map $\hproj:\LR{q}(\grp)^n\ra\LR{q}(\grp)^n$. One readily verifies that $\hproj$ is a projection,
and that $\Div\hproj f=0$. By Lemma \ref{lt_densitylemma}, $\hproj\LR{q}(\grp)^n\subset\LRsigma{q}(\grp)$ follows. On the other hand, 
since $\Div f=0$ implies $\xi_j\ft{f}_j=0$, we have $\hproj f=f$ for all $f\in\LRsigma{q}(\grp)$. Hence we conclude $\hproj\LR{q}(\grp)^n=\LRsigma{q}(\grp)$. 
\end{proof}

Since $\hproj:\LR{q}(\grp)^n\ra\LR{q}(\grp)^n$ is a continuous projection, it decomposes $\LR{q}(\grp)$ into a direct sum 
\begin{align}\label{lt_HelmholtzDecomp}
\LR{q}(\grp)=\LRsigma{q}(\grp)\oplus \gradspace{q}(\grp)
\end{align}
of closed subspaces with
\begin{align}\label{lt_gradspacedef}
\gradspace{q}(\grp) := \bp{\id-\hproj}\LR{q}(\grp)^n.
\end{align}

\subsection{Time-averaging}\label{lt_timeaveragingSection}

We shall now introduce the time-averaging projection \eqref{intro_defofproj} into the setting of vector fields defined on the group $\grp$. 
\begin{defn} We let 
\begin{align*} 
&\proj:\CRcisigma(\grp)\ra\CRcisigma(\grp),\quad \proj f(x,t):=\iper\int_0^\per f(x,s)\,\ds,\\
&\projcompl:\CRcisigma(\grp)\ra\CRcisigma(\grp),\quad \projcompl:=\id-\proj.
\end{align*}
\end{defn}

\begin{lem}\label{lt_projproplemma}
Let $q\in(1,\infty)$. The operator $\proj$ extends uniquely to a continuous projection
$\proj:\LRsigma{q}(\grp)\ra\LRsigma{q}(\grp)$ and  
$\proj:\WSRsigma{2,1}{q}(\grp)\ra\WSRsigma{2,1}{q}(\grp)$.
The same is true for $\projcompl$.
\end{lem}
\begin{proof}
Clearly, $\proj$ is a projection. 
Employing first Minkowski's integral inequality and then H\"older's inequality, we estimate 
\begin{align*}
\norm{\proj f}_{\LR{q}} &= \Bp{\iper\int_0^\per\int_{\R^n}\bigg|\iper\int_0^\per f(x,s)\,\ds\bigg|^q\,\dx\dt}^{1/q}\\
&= \Bp{\int_{\R^n}\bigg|\iper\int_0^\per f (x,s)\,\ds\bigg|^q\,\dx}^{1/q}\\
&\leq \iper\int_0^\per \Bp{ \int_{\R^n} \snorm{f(x,s)}^q\,\dx}^{1/q}\,\ds \leq \norm{f}_{\LR{q}}.
\end{align*}
Thus, by density $\proj$ extends uniquely to a continuous projection $\proj:\LRsigma{q}(\grp)\ra\LRsigma{q}(\grp)$. 
Estimating derivatives of $\proj f$ in the same manner, we find that $\proj$ is also bounded with respect to the $\WSR{2,1}{q}(\grp)$-norm. Consequently,
$\proj$ extends uniquely to a continuous projection $\proj:\WSRsigma{2,1}{q}(\grp)\ra\WSRsigma{2,1}{q}(\grp)$. 
It follows trivially that the same is true for $\projcompl$.  
\end{proof}

We use $\proj$ and $\projcompl$ to decompose $\LRsigma{q}(\grp)$ and $\WSRsigma{2,1}{q}(\grp)$ into direct sums of functions that are time-independent, \textit{i.e.}, steady states, and functions that have vanishing time-average. Put 
\begin{align}
&\LRsigmacompl{q}(\grp):=\projcompl\LRsigma{q}(\grp),\\
&\WSRsigmacompl{2,1}{q}(\grp):=\projcompl\WSRsigma{2,1}{q}(\grp).\label{lt_DefOfWSRsigmacomplo}
\end{align}
Identifying $\R^n$ as a subdomain of $\grp$, we observe:
\begin{lem}\label{lt_decomplemma}
Let $q\in(1,\infty)$. The projection $\proj$ induces the decompositions
\begin{align}
&\LRsigma{q}(\grp)=\LRsigma{q}(\R^n)\oplus\LRsigmacompl{q}(\grp),\label{lt_decomplemmaLRsigmadecompl}\\
&\WSRsigma{2,1}{q}(\grp)=\WSRsigma{2}{q}(\R^n)\oplus\WSRsigmacompl{2,1}{q}(\grp).\label{lt_decomplemmaAnIsoWSRdecompl}
\end{align}
\end{lem}
\begin{proof}
Since, by Lemma \ref{lt_projproplemma}, $\proj:\LRsigma{q}(\grp)\ra\LRsigma{q}(\grp)$ is a (continuous) projection, it follows that
$\LR{q}(\grp)=\proj\LR{q}(\grp)\oplus\projcompl\LR{q}(\grp)$.
Consequently, to show \eqref{lt_decomplemmaLRsigmadecompl} we only need to verify that $\proj\LRsigma{q}(\grp)=\LRsigma{q}(\R^n)$. 
This, however, is an easy consequence of the fact that $\proj f$ is 
independent on $t$ and thus $\norm{\proj f}_{\LR{q}(\R^n)}=\norm{\proj f}_{\LR{q}(\grp)}$. The decomposition \eqref{lt_decomplemmaAnIsoWSRdecompl} follows analogously.
\end{proof}

Next, we compute the symbols of the projections $\proj$ and $\projcompl$ with respect to the Fourier transform on $\grp$. 
\begin{lem}\label{lt_projsymbollem} 
For $f\in\SR(\grp)$
\begin{align}
&\proj f = \iFT_\grp\bb{\projsymbol\cdot \ft{f}},\label{lt_projsymbollemproj}\\
&\projcompl f = \iFT_\grp\bb{(1-\projsymbol)\cdot \ft{f}}\label{lt_projsymbollemprojcompl}
\end{align}
with
\begin{align*}
\projsymbol:\dualgrp\ra\CNumbers,\quad
\projsymbol(\xi,k):=
\begin{pdeq}
&1 && \tif k=0,\\
&0 && \tif k\neq0.
\end{pdeq}
\end{align*}
\end{lem}
\begin{proof}
We simply observe that
\begin{align*}
\FT_\grp\bb{\proj f}(\xi,k) &= \iper\int_0^\per\int_{\R^n}\iper\int_0^\per f(x,s)\,\ds\,e^{-ix\cdot\xi-i\perf k t}\,\dx\dt\\
&= \projsymbol(\xi,k) \int_{\R^n}\iper\int_0^\per f(x,s)\,\ds\,e^{-ix\cdot\xi}\,\dx\\
&= \projsymbol(\xi,k)\, \ft{f}(\xi,0) = \projsymbol(\xi,k)\, \ft{f}(\xi,k).
\end{align*}
\end{proof}

\subsection{Reformulation}

Since the topology and differentiable structure on $\grp$ is inherited from $\R^n\times\R$, we obtain the following equivalent 
formulation of the time-periodic problem \eqref{intro_linnssystem}, including the periodicity conditions \eqref{intro_timeperiodicsolution}--\eqref{intro_timeperiodicdata}, as a system over $\grp$-defined vector fields:
\begin{align}\label{lt_linnspastbodywholespaceregrpformulation}
\begin{pdeq}
&\partial_t\uvel -\Delta\uvel -\rey\partial_1\uvel + \grad\upres = f && \tin\grp,\\
&\Div\uvel =0 && \tin\grp
\end{pdeq}
\end{align} 
with unknowns $\uvel:\grp\ra\R^n$ and $\upres:\grp\ra\R$, and data $f:\grp\ra\R^n$.
Observe that in this formulation the periodicity conditions are 
not needed anymore. 
Indeed, all functions defined on $\grp$ are by definition $\per$-time-periodic.

Based on the new formulation above, we also obtain the following new formulation of Theorem \ref{mainresult_tpmaxregthm}:

\begin{thm}\label{lt_TPOseenMappingThm}
Let $q\in(1,\infty)$. Put $\ALTP:=\partial_t -\Delta -\rey\partial_1$.
Then 
\begin{align*}
\ALTP:\ \WSRsigmacompl{2,1}{q}(\grp)\ra\LRsigmacompl{q}(\grp)
\end{align*}
homeomorphically. Moreover
\begin{align}\label{lt_TPOseenMappingThmInversebound}
\norm{\ALTPinverse}\leq \Cc{C}\,\polynomial(\rey,\per),
\end{align}
where $\Cclast{C}=\Cclast{C}(q,n)$ and $\polynomial(\rey,\per)$ is a polynomial in $\rey$ and $\per$.
\end{thm}

The main challenge is now to prove Theorem \ref{lt_TPOseenMappingThm}. This will be done in the next section. At this point we just emphasize the crucial advantage obtained by formulating the problem in a group setting, which is the ability by means of the Fourier transform $\FT_\grp$ 
to express the solution $\uvel$ to \eqref{lt_linnspastbodywholespaceregrpformulation} in terms of a Fourier multiplier. 
If we namely apply the Fourier transform $\FT_\grp$ on both sides of the equations in \eqref{lt_linnspastbodywholespaceregrpformulation}, we obtain the equivalent
system\footnote{We make use of the Einstein summation convention and
implicitly sum over all repeated indices.}
\begin{align*}
\begin{pdeq}
&(i\perf k) \uvelft + \snorm{\xi}^2\uvelft -\rey i\xi_1 \uvelft + i\,\upresft\,\xi = \ft{f} && \tin\dualgrp,\\
&\xi_j\uvelft_j =0 && \tin\dualgrp.
\end{pdeq}
\end{align*} 
Dot-multiplying the first equation with $\xi$, we obtain the relation $i\, \upresft\snorm{\xi}^2 = \xi_j\ft{f}_j$ 
and thus
\begin{align*}
\Bp{(i\perf k) + \snorm{\xi}^2 -\rey i\xi_1} \uvelft = \Bp{\idmatrix-\frac{\xi\otimes\xi}{\snorm{\xi}^2}}\ft{f}.
\end{align*}
Formally at least, we can therefore deduce
\begin{align}\label{lt_fouriermultiplierrepresentation}
\uvel = \iFT_\grp\Bb{ \frac{1}{(i\perf k) + \snorm{\xi}^2 -\rey i\xi_1}\ft{\hproj f}}.
\end{align}
This representation formula for the solution $\uvel$ shall play a central role in the proof of Theorem \ref{lt_TPOseenMappingThm} presented in the next section.

\section{Proof of main theorems}

The main tool in the proof of Theorem \ref{lt_TPOseenMappingThm} is the following theorem on transference of Fourier multipliers, 
which enables us to ``transfer'' multipliers from one group setting into another.
The theorem is originally due to \textsc{de Leeuw} \cite{Leeuw1965}, who established the transference principle between the torus group and $\R$. 
The more general version below is due to \textsc{Edwards} and \textsc{Gaudry} \cite[Theorem B.2.1]{EdwardsGaudryBook}.

\begin{thm}\label{lt_transferenceofmultipliersThm}
Let $\grp$ and $\grpH$ be locally compact abelian groups. Moreover, let 
\begin{align*}
\Phi:\dualgrp\ra\dualgrpH
\end{align*}
be a continuous homomorphism and $q\in[1,\infty]$. Assume that $m\in\LR{\infty}(\dualgrpH;\CNumbers)$ is a continuous $\LR{q}$-multiplier, that is, there is 
a constant $\Cc[lt_transferenceofmultipliersThmConst1]{C}$ such that 
\begin{align*}
\forall f \in\LR{2}(\grpH)\cap\LR{q}(\grpH):\ \norm{\iFT_\grpH\bb{m\cdot \FT_\grpH(f)}}_q\leq
\const{lt_transferenceofmultipliersThmConst1}\norm{f}_q.
\end{align*}
Then $m\circ\Phi\in\LR{\infty}(\dualgrp;\CNumbers)$ is also an $\LR{q}$-multiplier with
\begin{align*}
\forall f \in\LR{2}(\grp)\cap\LR{q}(\grp):\ \norm{\iFT_\grp\bb{m\circ\Phi\cdot \FT_\grp(f)}}_q\leq
\const{lt_transferenceofmultipliersThmConst1}\norm{f}_q. 
\end{align*}
\end{thm}

\begin{rem}
We shall employ Theorem \ref{lt_transferenceofmultipliersThm} with $\grpH:=\R^n\times\R$ and $\grp:=\R^n\times\R/\per\Z$. A proof
of the theorem for this particular choice of groups can be found in \cite[Theorem 3.4.5]{habil}. 
\end{rem}

We are now in a position to prove Theorem \ref{lt_TPOseenMappingThm}.
\begin{proof}[Proof of Theorem \ref{lt_TPOseenMappingThm}]\newCCtr[c]{lt_TPOseenMappingThm}\newCCtr[P]{lt_TPOseenMappingThmPolynomial}
By construction, $\ALTP$ is a bounded mapping from $\WSRsigma{2,1}{q}(\grp)$ into $\LRsigma{q}(\grp)$. 
As one may easily verify, the diagram
\begin{align*}
\begin{CD}
\WSRsigma{2,1}{q}(\grp) @>\ALTP>> \LRsigma{q}(\grp)\\
@V\projcompl VV @VV \projcompl V\\
\WSRsigma{2,1}{q}(\grp) @>\ALTP>> \LRsigma{q}(\grp)
\end{CD}
\end{align*}
commutes. Thus also
\begin{align}\label{lt_TPOseenMappingThmbasicmapping}
\ALTP:\WSRsigmacompl{2,1}{q}(\grp)\ra\LRsigmacompl{q}(\grp)
\end{align}
is a bounded map. 

We shall now show that the mapping \eqref{lt_TPOseenMappingThmbasicmapping} is onto. To this end, consider
first a vector field $f\in\projcompl\CRcisigma(\grp)$. Clearly, $f\in\SR(\grp)^n$. In view of \eqref{lt_fouriermultiplierrepresentation}, we put
\begin{align*}
\uvel := \iFT_\grp\Bb{ \frac{1}{i\perf k + \snorm{\xi}^2 -\rey i\xi_1}\ft{f}}.
\end{align*}
At the outset, it is not clear that $\uvel$ is well-defined. However, since $f=\projcompl f$ we recall \eqref{lt_projsymbollemprojcompl} 
to deduce $\ft{f}=\bp{1-\projsymbol}\ft{f}$ and thus
\begin{align}\label{lt_TPOseenMappingThmMmultiplieruvel}
\uvel = \iFT_\grp\Bb{ \frac{\bp{1-\projsymbol(\xi,k)}}{i\perf k + \snorm{\xi}^2 -\rey i\xi_1}\ft{f}} = \iFT_\grp\bb{\Mmultiplier(\xi,k)\cdot \ft{f}}
\end{align}
with 
\begin{align}\label{lt_TPOseenMappingThmmmultiplierdef}
\Mmultiplier:\dualgrp\ra\CNumbers,\quad \Mmultiplier(\xi,k):=\frac{1-\projsymbol(\xi,k)}{\snorm{\xi}^2 + i(\perf k -\rey \xi_1)}.
\end{align}
Observe that the denumerator $\snorm{\xi}^2 + i(\perf k -\rey \xi_1)$ in the definition \eqref{lt_TPOseenMappingThmmmultiplierdef} of $\Mmultiplier$ vanishes only at $(\xi,k)=(0,0)$. 
Since the numerator $1-\projsymbol(\xi,k)$ in \eqref{lt_TPOseenMappingThmmmultiplierdef} vanishes in a neighborhood around $(0,0)$, we see
that $M\in\LR{\infty}(\dualgrp;\CNumbers)$. It therefore follows from \eqref{lt_TPOseenMappingThmMmultiplieruvel} that $\uvel$ is well-defined as an element of 
$\TDR(\grp)$. 
It follows directly from the definition of $\uvel$ that $\ALTP\uvel = f$.
The challenge is now to show that $\uvel\in\WSRsigmacompl{2,1}{q}(\grp)$ and establish the estimate
\begin{align}\label{lt_TPOseenMappingThmbasicestimate}
\norm{\uvel}_{2,1,q} \leq c\, \norm{f}_q
\end{align}
for some constant $c$.
We shall use the transference principle for multipliers, that is, Theorem \ref{lt_transferenceofmultipliersThm}, to establish 
\eqref{lt_TPOseenMappingThmbasicestimate}. For this purpose, let $\chi$ be a ``cut-off'' function with
\begin{align*}
\chi\in\CRci(\R;\R),\quad \chi(\eta)=1\ \text{for}\ \snorm{\eta}\leq\half,\quad \chi(\eta)=0\ \text{for}\ \snorm{\eta}\geq 1.
\end{align*}
We then define
\begin{align}\label{lt_TPOseenMappingThmMmultiplier}
\mmultiplier:\R^n\times\R\ra\CNumbers,\quad \mmultiplier(\xi,\eta):=\frac{1-\chi(\iperf\eta)}{\snorm{\xi}^2 + i(\eta -\rey \xi_1)}.
\end{align}
We can consider $\R^n\times\R$ as a group $\grpH$ with addition as group operation. Endowed with the Euclidean topological, $\grpH$ becomes 
a locally compact abelian group. It is well-known that the dual group $\dualgrpH$ can also be identified with $\R^n\times\R$ equipped with the 
Euclidean topology. We can thus consider $\mmultiplier$ as mapping
$\mmultiplier:\dualgrpH\ra\CNumbers$.
In order to employ Theorem \ref{lt_transferenceofmultipliersThm}, we define 
$\Phi:\dualgrp\ra\dualgrpH$, $\Phi(\xi,k):=(\xi,\perf k)$.
Clearly, $\Phi$ is a continuous homomorphism. Moreover, $\Mmultiplier = \mmultiplier\circ\Phi$.
Consequently, if we can show that $\mmultiplier$ is a continuous $\LR{q}(\R^n\times\R)$-multiplier we may conclude from Theorem \ref{lt_transferenceofmultipliersThm} that $\Mmultiplier$ is an $\LR{q}(\grp)$-multiplier. 
We first observe that the only zero of the denumerator 
$\snorm{\xi}^2 + i(\eta -\rey \xi_1)$
in definition \eqref{lt_TPOseenMappingThmMmultiplier} of $\mmultiplier$ is $(\xi,\eta)=(0,0)$.
Since the numerator $1-\chi\bp{\iperf\eta}$ in \eqref{lt_TPOseenMappingThmMmultiplier} vanishes in a neighborhood of $(0,0)$, we see that $\mmultiplier$ is continuous; in fact $\mmultiplier$ is smooth.  
We shall now apply Marcinkiewicz's multiplier theorem, see for example \cite[Corollary 5.2.5]{Grafakos1} or \cite[Chapter IV, \S 6]{Stein70}, to show that $\mmultiplier$ is an $\LR{q}(\R^n\times\R)$-multiplier. Note that Marcinkiewicz's multiplier theorem
can be employed at this point since $\mmultiplier$ is a Fourier multiplier in the Euclidean $\R^n\times\R$ setting.
To employ Marcinkiewicz's multiplier theorem, we must verify that 
\begin{align}\label{lt_TPOseenMappingThmMarcinkiewczCond}
\sup_{\epsilon\in\set{0,1}^{n+1}}\sup_{(\xi,\eta)\in\R^n\times\R} \snorml{\xi_1^{\epsilon_1}\cdots\xi_n^{\epsilon_n}\eta^{\epsilon_{n+1}}
\partial_{1}^{\epsilon_1}\cdots\partial_{n}^{\epsilon_n}\partial_\eta^{\epsilon_{n+1}}
\mmultiplier(\xi,\eta)} \leq \Cc[lt_TPOseenMappingThmConst1]{lt_TPOseenMappingThm}.
\end{align}
Since $\mmultiplier$ is smooth, \eqref{lt_TPOseenMappingThmMarcinkiewczCond} follows if we can show that all functions of type
\begin{align*}
(\xi,\eta)\ra
\xi_1^{\epsilon_1}\cdots\xi_n^{\epsilon_n}\eta^{\epsilon_{n+1}}
\partial_{1}^{\epsilon_1}\cdots\partial_{n}^{\epsilon_n}\partial_\eta^{\epsilon_{n+1}}
\mmultiplier(\xi,\eta)
\end{align*}
stay bounded as $\snorm{(\xi,\eta)}\ra\infty$. Since $\mmultiplier$ is a rational function with non-vanishing denumerator away from $(0,0)$, this is easy to verify.
Consequently, we conclude \eqref{lt_TPOseenMappingThmMarcinkiewczCond}. 
If we analyze the bound on the functions more carefully,
we find that $\const{lt_TPOseenMappingThmConst1}$ can be chosen such that
$\const{lt_TPOseenMappingThmConst1} = \Cc[lt_TPOseenMappingThmPolynomial1]{lt_TPOseenMappingThmPolynomial}(\rey,\per)$
with $\const{lt_TPOseenMappingThmPolynomial1}(\rey,\per)$ a polynomial in $\rey$ and $\per$.
By Marcinkiewicz's multiplier theorem, see for example \cite[Corollary 5.2.5]{Grafakos1} or \cite[Chapter IV, \S 6]{Stein70},
$\mmultiplier$ is an $\LR{q}(\R^n\times\R)$-multiplier. 
We now recall Theorem \ref{lt_transferenceofmultipliersThm} and conclude that $\Mmultiplier=\mmultiplier\circ\Phi$ is an $\LR{q}(\grp)$-multiplier. 
Since $\uvel=\iFT_\grp\bb{\Mmultiplier(\xi,k)\cdot \ft{f}}$, we thus obtain
\begin{align}\label{lt_TPOseenMappingThmApriori1}
\norm{\uvel}_q\leq \Cc{lt_TPOseenMappingThm}\,\const{lt_TPOseenMappingThmPolynomial1}(\rey,\per)\,\norm{f}_q,
\end{align}
with $\Cclast{lt_TPOseenMappingThm}=\Cclast{lt_TPOseenMappingThm}(q,n)$. 
Differentiating $\uvel$ with respect to time and space, we further obtain from the equation $\uvel=\iFT_\grp\bb{\Mmultiplier(\xi,k)\cdot \ft{f}}$ the identities
\begin{align*}
\partial_t\uvel=\iFT_\grp\Bb{(i\perf k)\,\Mmultiplier(\xi,k)\cdot \ft{f}}
\end{align*}
and
\begin{align*}
\partial_x^\alpha\uvel=\iFT_\grp\Bb{(i\xi)^\alpha\,\Mmultiplier(\xi,k)\cdot \ft{f}},
\end{align*}
respectively.
We can now repeat the argument above with $(i\perf k)\Mmultiplier(\xi,k)$ in the role of the multiplier $\Mmultiplier$,
and $(i\perf \eta)\mmultiplier(\xi,\eta)$ in the role of $\mmultiplier$, to conclude
\begin{align}\label{lt_TPOseenMappingThmApriori2}
\norm{\partial_t\uvel}_q\leq \Cc{lt_TPOseenMappingThm}\,\Cc{lt_TPOseenMappingThmPolynomial}(\rey,\per)\,\norm{f}_q,
\end{align}
with $\Cclast{lt_TPOseenMappingThm}=\Cclast{lt_TPOseenMappingThm}(q,n)$.
Similarly, for $\snorm{\alpha}\leq 2$ we repeat the argument above with $(i\xi)^\alpha\Mmultiplier(\xi,k)$ in the role of $\Mmultiplier$,
and $(i\xi)^\alpha\mmultiplier(\xi,\eta)$ in the role of $\mmultiplier$, and obtain
\begin{align}\label{lt_TPOseenMappingThmApriori3}
\norm{\partial_x^\alpha\uvel}_q\leq \Cc[lt_TPOseenMappingThmApriori3const]{lt_TPOseenMappingThm}\,\Cc[lt_TPOseenMappingThmApriori3poly]{lt_TPOseenMappingThmPolynomial}(\rey,\per)\,\norm{f}_q.
\end{align}
Collecting \eqref{lt_TPOseenMappingThmApriori1}, \eqref{lt_TPOseenMappingThmApriori2} and \eqref{lt_TPOseenMappingThmApriori3}, we thus conclude
\begin{align}\label{lt_TPOseenMappingThmbasicestimateFinal}
\norm{\uvel}_{2,1,q} \leq \Cc{lt_TPOseenMappingThm}\,\Cc{lt_TPOseenMappingThmPolynomial}(\rey,\per) \norm{f}_q,
\end{align}
with $\Cclast{lt_TPOseenMappingThm}=\Cclast{lt_TPOseenMappingThm}(q,n)$. 
Since $\hproj f=f$, we see directly from \eqref{lt_TPOseenMappingThmMmultiplieruvel} that $\hproj\uvel=\uvel$ and thus $\Div \uvel=0$.
Clearly, also $\projcompl\uvel=\uvel$. Recalling \eqref{lt_densitylemmaAnisoWSRsigmaCharacterization} and \eqref{lt_DefOfWSRsigmacomplo}, it follows that
$\uvel\in\WSRsigmacompl{2,1}{q}(\grp)$.
Consequently, we have constructed for arbitrary $f\in\projcompl\CRcisigma(\grp)$ a vector field $\uvel\in\WSRsigmacompl{2,1}{q}(\grp)$ such that 
$\ALTP\uvel=f$ and for which \eqref{lt_TPOseenMappingThmbasicestimateFinal} holds. Since $\CRcisigma(\grp)$ is a dense subset of $\LRsigma{q}(\grp)$,
it follows that $\projcompl\CRcisigma(\grp)$ is dense in $\LRsigmacompl{q}(\grp)$. Thus, by a standard density argument we can find for
any $f\in\LRsigmacompl{q}(\grp)$ a vector field $\uvel\in\WSRsigmacompl{2,1}{q}(\grp)$ that satisfies 
$\ALTP\uvel=f$ and \eqref{lt_TPOseenMappingThmbasicestimateFinal}. 
In particular, we have verified that the mapping \eqref{lt_TPOseenMappingThmbasicmapping} is onto.

Finally, we must verify that the mapping \eqref{lt_TPOseenMappingThmbasicmapping} is injective. 
Consider therefore a vector field $\uvel\in\WSRsigmacompl{2,1}{q}(\grp)$ with $\ALTP \uvel=0$. 
Employing the Fourier transform $\FT_\grp$, we then deduce
$\bp{i\perf k + \snorm{\xi}^2 -\rey i\xi_1} \uvelft = 0$.
Since the polynomial $\snorm{\xi}^2 + i\bp{\perf k -\rey \xi_1}$ vanishes only at $(\xi,k)=(0,0)$,
we conclude that $\supp\ft{\uvel}\subset\set{(0,0)}$. However, since 
$\proj\uvel=0$ we have $\projsymbol\ft{\uvel}=0$, whence $(\xi,0)\notin\supp\ft{\uvel}$ for all $\xi\in\R^n$. Consequently,
$\supp\ft{\uvel}=\emptyset$. It follows that $\ft{\uvel}=0$ and thus $\uvel=0$. We conclude that the mapping \eqref{lt_TPOseenMappingThmbasicmapping} is injective. 

Since the mapping \eqref{lt_TPOseenMappingThmbasicmapping} is bounded and bijective, it is a homeomorphism by the open mapping theorem.
The bound \eqref{lt_TPOseenMappingThmInversebound} follows from \eqref{lt_TPOseenMappingThmbasicestimateFinal}. 
\end{proof}

\begin{rem}
In the proof above, it is crucial that the numerator $1-\chi(\iperf\eta)$ in the fraction that defines $\mmultiplier$ in \eqref{lt_TPOseenMappingThmMmultiplier}
vanishes in a neighborhood of the only zero of the denumerator ${\snorm{\xi}^2 + i(\eta -\rey \xi_1)}$.
Consequently, $\mmultiplier$ is a bounded and even smooth multiplier. The key reason $\mmultiplier$ can be chosen with this structure 
is that the data $f\in\LRsigmacompl{q}(\grp)$ is $\projcompl$-invariant, that is, $\projcompl f=f$. In other words, we would not be able to carry out the argument for a 
general $f\in\LRsigma{q}(\grp)$. This observation illustrates the necessity of the decomposition induced by the projections $\proj$ and $\projcompl$.     
\end{rem}

\begin{proof}[Proof of Theorem \ref{mainresult_tpmaxregthm}]
The statements in Theorem \ref{mainresult_tpmaxregthm} were established in Theorem \ref{lt_TPOseenMappingThm} for the system \eqref{lt_linnspastbodywholespaceregrpformulation}.
To prove Theorem \ref{mainresult_tpmaxregthm}, we therefore only need to verify that \eqref{lt_linnspastbodywholespaceregrpformulation} is fully equivalent
to \eqref{intro_linnssystem}--\eqref{intro_timeperiodicdata}. The verification can be reduced to three simple observations.
We first observe that $\bijection$ induces, by lifting, an 
isometric isomorphism between $\WSRsigmacompl{2,1}{q}(\grp)$ and $\WSRsigmapercompl{2,1}{q}\bp{\R^n\times(0,\per)}$.
We next observe that the trivial extension of a function $F:\R^n\times[0,\per)\ra\R$ to a time-periodic function on $\R^n\times\R$ 
is given by $F\circ(\bijectioninv\circ\quotientmap)$. Finally, we observe that $\quotientmap$ induces, by lifting, an embedding of
$\WSRsigmacompl{2,1}{q}(\grp)$ into the subspace
\begin{align*}
\setc{\uvel\in\LRloc{1}(\R^n\times\R)}{\partial_t^\beta\partial_x^\alpha\uvel \in \LRloc{1}(\R^n\times\R)\text{ for }\snorm{\beta}\leq 1,\,\snorm{\alpha}\leq 2}
\end{align*}
of distributions $\DDR\np{\R^n\times\R}$. Consequently, recalling \eqref{lt_defofgrpderivatives}, we deduce for any element $\uvel\in\WSRsigmacompl{2,1}{q}(\grp)$ that
$\partial_t^\beta\partial_x^\alpha\uvel = \bb{\partial_t^\beta\partial_x^\alpha (\uvel\circ\quotientmap)}\circ\bijectioninv$ for $\snorm{\alpha}\leq 2$, $\snorm{\beta}\leq 1$ with the derivatives $\partial_t^\beta\partial_x^\alpha (\uvel\circ\quotientmap)$ being understood in the sense of distributions and  
thereby, by the embedding above, as functions in $\LRloc{1}(\R^n\times\R)$.

Now consider a vector field $f\in\LRsigmacompl{q}\bp{\R^n\times(0,\per)}$. Clearly, $f\circ\bijectioninv\in\LRsigmacompl{q}(\grp)$. 
Recall Theorem \ref{lt_TPOseenMappingThm} and put $\tuvel:=\ALTPinverse\bp{f\circ\bijectioninv}\in\WSRsigmacompl{2,1}{q}(\grp)$. 
Then
$\uvel:=\tuvel\circ\bijection\in\WSRsigmapercompl{2,1}{q}\bp{\R^n\times(0,\per)}$ with
\begin{align*}
\norm{\uvel}_{2,1,q} = \norm{\tuvel}_{2,1,q} \leq \norm{\ALTPinverse} \norm{f\circ\bijectioninv}_q = \norm{\ALTPinverse} \norm{f}_q.
\end{align*}
Recalling \eqref{lt_TPOseenMappingThmInversebound}, we see that $\uvel$ satisfies \eqref{mainresult_tpmaxregthmEstimate}.
Moreover, since $\ALTP \tuvel = f\circ\bijectioninv$ it follows that
$\bb{\ALTP(\tuvel\circ\quotientmap)}\circ\bijectioninv = f\circ\bijectioninv$, from which we can easily deduce
$\ALTP\bp{\uvel\circ(\bijectioninv\circ\quotientmap)}=f\circ(\bijectioninv\circ\quotientmap)$. The latter identity shows that the
trivial extensions of $\uvel$ and $f$ to time-periodic vector fields on $\R^n\times\R$ satisfy $\eqref{intro_linnssystem}$ with $\upres=0$.

It remains to verify uniqueness of the solution $\uvel$. If, however, $\ouvel\in\WSRsigmapercompl{2,1}{q}\bp{\R^n\times(0,\per)}$ is another solution, then 
$(\uvel-\ouvel)\circ\bijectioninv\in\WSRsigmacompl{2,1}{q}(\grp)$ with $\ALTP\bb{(\uvel-\ouvel)\circ\bijectioninv}=0$. Hence $\uvel-\ouvel=0$ by the injectivity
of $\ALTP$ established in Theorem \ref{lt_TPOseenMappingThm}.
\end{proof}

We proceed with the proof of Theorem \ref{mainresult_maxregthm}. In order to characterize the pressure term in \eqref{intro_linnssystem}, we need
the following lemma:

\begin{lem}\label{lt_PressureMappingLem}
Let $q\in(1,n)$. Put
\begin{align}\label{lt_PressureMappingLemPressureSpace}
\begin{aligned}
&\xpres{q}(\grp):=\setc{\upres\in\TDR(\grp)\cap\LRloc{1}(\grp)}{\norm{\upres}_{\xpres{q}}<\infty},\\
&\norm{\upres}_{\xpres{q}}:=\bigg( \iper\int_0^\per \norm{\upres(\cdot,t)}_{\frac{nq}{n-q}}^q\,\dt + \norm{\grad\upres}_q^q\bigg)^{1/q}.
\end{aligned}
\end{align}
Then 
\begin{align}\label{lt_PressureMappingLemGradmap}
\begin{aligned}
\gradmap: \xpres{q}(\grp)\ra \gradspace{q}(\grp),\quad \gradmap\,\upres:=\grad\upres
\end{aligned}
\end{align}
is a homeomorphism. Moreover, $\norm{\gradmapinverse}$ depends only on $n$ and $q$. 
\end{lem}
\begin{proof}\newCCtr[c]{lt_PressureMappingLem}
Clearly, $\gradmap$ is bounded. Consider $\upres\in\ker\gradmap$. Then $\grad\upres=0$ and it thus follows by standard arguments that 
$\upres(x,t)=c(t)$. Since $\norm{\upres}_{\xpres{q}}<\infty$, we must have $\upres=0$. Consequently, $\gradmap$ is injective.
To show that $\gradmap$ is onto, we consider the mapping
\begin{align*}
\cali:\SR(\grp)^n\ra\TDR(\grp),\quad \cali(f):=\iFT_{\R^n}\Bb{\frac{\xi_j}{\snorm{\xi}^2}\cdot \FT_{\R^n}\bp{f_j}},
\end{align*}
where $\FT_{\R^n}:\SR(\grp)\ra\SR(\grp)$ denotes the partial Fourier transform. Observe that
\begin{align}\label{lt_PressureMappingLemInvProp}
\grad\cali(f) = \bp{\id-\hproj} f.
\end{align}
Since $\cali$ can expressed as a Riesz potential composed with a sum of Riesz operators, 
well-known properties of the Riesz potential (see for example \cite[Theorem 6.1.3]{Grafakos2}) and the Riesz operators 
(see for example \cite[Corollary 4.2.8]{Grafakos1}) yield
\begin{align}\label{lt_PressureMappingLemPartialEstOfI}
\iper\int_0^\per \norm{\cali(f)(\cdot,t)}_{\frac{nq}{n-q}}^q\,\dt\leq 
\Cc{lt_PressureMappingLem} \iper\int_0^\per \norm{ f(\cdot,t)  }_{q}^q\,\dt = \Cclast{lt_PressureMappingLem}\,\norm{f}_q^q
\end{align}
with $\Cclast{lt_PressureMappingLem}=\Cclast{lt_PressureMappingLem}(n,q)$.
In combination with \eqref{lt_PressureMappingLemInvProp}, \eqref{lt_PressureMappingLemPartialEstOfI} implies
$\norm{\cali(f)}_{\xpres{q}} \leq \Cc{lt_PressureMappingLem}\, \norm{f}_q$.
By a density argument, we can extend $\cali$ uniquely to a bounded map 
$\cali:\LR{q}(\grp)^n\ra\xpres{q}(\grp)$
that satisfies \eqref{lt_PressureMappingLemInvProp} for all $f\in\LR{q}(\grp)^n$. We can now show that $\gradmap$ is onto. 
If namely $f\in\gradspace{q}(\grp)$, then $\grad \cali(f)=f$. We conclude by the open mapping theorem that
$\gradmap$ is a homeomorphism. In fact, the inverse is given by $\cali$, whence by \eqref{lt_PressureMappingLemPartialEstOfI} $\norm{\gradmapinverse}$ depends
only on $n$ and $q$.
\end{proof}

\begin{proof}[Proof of Theorem \ref{mainresult_maxregthm}]\newCCtr[c]{mainresult_maxregthm}
Let $f\in\LR{q}\bp{\R^n\times(0,\per)}^n$. We put $\tf:=f\circ\bijectioninv\in\LR{q}(\grp)^n$ and employ the Helmholtz projection to decompose 
$\tf=\hproj \tf + (\id-\hproj)\tf$. 
By Lemma \ref{lt_HelmholtzProjLem}, $\hproj \tf\in\LRsigma{q}(\grp)$. 
We further decompose $\hproj \tf=\proj\hproj \tf+\projcompl\hproj\tf$ and recall from Lemma \ref{lt_decomplemma} that $\proj\hproj \tf\in\LRsigma{q}(\R^n)$ and 
$\projcompl\hproj \tf\in\LRsigmacompl{q}(\grp)$. 

Well-known theory for the linearized, steady-state Navier-Stokes system, see for example
\cite[Theorem IV.2.1]{galdi:book1} and
\cite[Theorem VII.4.1]{galdi:book1}, implies existence of a unique solution $\vvel\in\maxregspacelinnss{q}(\R^n)$ to 
$-\Delta\vvel-\rey\partial_1\vvel=\proj\hproj \tf$
that satisfies 
\begin{align}\label{mainresult_maxregthmProofEstOfvvel}
\norm{\vvel}_{\maxregspacelinnss{q}} \leq \Cc{mainresult_maxregthm}\norm{\proj\hproj\tf}_q\leq \Cclast{mainresult_maxregthm}\norm{f}_q,
\end{align}
with $\Cclast{mainresult_maxregthm}=\Cclast{mainresult_maxregthm}(n,q)$. Here,
$\maxregspacelinnss{q}(\R^n)$ is defined by \eqref{mainresult_maxregthmDefOfXs1}, \eqref{mainresult_maxregthmDefOfXs2} or \eqref{mainresult_maxregthmDefOfXs3} depending on the choice of $n,\rey,q$. 

Recalling Theorem \ref{lt_TPOseenMappingThm}, we put $\twvel:=\ALTPinverse\projcompl\hproj\tf\in\WSRsigmacompl{2,1}{q}(\grp)$.
As in the proof of Theorem \ref{mainresult_tpmaxregthm}, we then observe that $\wvel:=\twvel\circ\bijection\in\WSRsigmapercompl{2,1}{q}\bp{\R^n\times(0,\per)}$ and 
that the trivial extension of $\wvel$ to a time-periodic vector field on $\R^n\times\R$ is given by $\wvel\circ(\bijectioninv\circ\quotientmap)$ and satisfies
$\ALTP\bb{\wvel\circ(\bijectioninv\circ\quotientmap)}=\bb{\projcompl\hproj\tf}\circ\quotientmap$ in the sense of distributions.
Moreover, using \eqref{lt_TPOseenMappingThmInversebound} we can estimate
\begin{align}\label{mainresult_maxregthmProofEstOfwvel}
\norm{\wvel}_{2,1,q} = \norm{\twvel}_{2,1,q} \leq \norm{\ALTPinverse} \norm{\projcompl\hproj\tf}_q \leq \Cc{mainresult_maxregthm} \norm{f}_q 
\end{align}
with $\Cclast{mainresult_maxregthm}=\Cclast{mainresult_maxregthm}(\rey,\per,q,n)$.

We now put $\uvel:=\vvel+\wvel$. We then have $u\in\maxregspacelinnss{q}(\R^n)\oplus\WSRsigmapercompl{2,1}{q}\bp{\R^n\times(0,\per)}$ and 
$\ALTP\bb{\uvel\circ(\bijectioninv\circ\quotientmap)}=\bb{\hproj\tf}\circ\quotientmap$ in the sense of distributions.

Finally, we recall Lemma \ref{lt_PressureMappingLem} and put $\tupres:=\gradmapinverse\bb{\np{\id-\hproj}\tf}\in\xpres{q}(\grp)$.
One readily verifies that $\bijection$ induces, by lifting, an isometric isomorphism between $\xpres{q}(\grp)$ and $\xpres{q}\bp{\R^n\times(0,\per)}$.
Thus $\upres:=\tupres\circ\bijection\in\xpres{q}\bp{\R^n\times(0,\per)}$. Since $\grad\tupres=\np{\id-\hproj}\tf$, 
if follows by the same observations as those made in the proof of Theorem \ref{mainresult_tpmaxregthm} that 
$\grad\bb{\upres\circ(\bijectioninv\circ\quotientmap)}=\bb{(\id-\hproj)\tf}\circ\quotientmap$ in the sense of distributions. 
Moreover, we can estimate
\begin{align}\label{mainresult_maxregthmProofEstOfupres}
\norm{\upres}_{\xpres{q}} = \norm{\tupres}_{\xpres{q}} \leq \norm{\gradmapinverse} \norm{\np{\id-\hproj}\tf}_q \leq \Cc{mainresult_maxregthm} \norm{f}_q 
\end{align}
with $\Cclast{mainresult_maxregthm}=\Cclast{mainresult_maxregthm}(q,n)$.

We conclude that $\ALTP\bb{\uvel\circ(\bijectioninv\circ\quotientmap)}+\grad\bb{\upres\circ(\bijectioninv\circ\quotientmap)}=f\circ(\bijectioninv\circ\quotientmap)$
in the sense of distributions.
Since both $\vvel$ and $\wvel$ are solenoidal vector fields, we see that also $\Div\bb{\uvel\circ(\bijectioninv\circ\quotientmap)}=0$. 
We have thus shown that the trivial extension of $(\uvel,\upres)$ and $f$ to time-periodic vector fields on $\R^n\times\R$ is a solution to 
\eqref{intro_linnssystem}--\eqref{intro_timeperiodicdata}. Furthermore, by \eqref{mainresult_maxregthmProofEstOfvvel}, \eqref{mainresult_maxregthmProofEstOfwvel} and 
\eqref{mainresult_maxregthmProofEstOfupres} we have established \eqref{mainresult_maxregthmEst}.

It remains to show uniqueness of the solution. Assume 
\begin{align*}
(\ouvel,\oupres)\in\Bp{\maxregspacelinnss{q}(\R^n)\oplus\WSRsigmapercompl{2,1}{q}\bp{\R^n\times(0,\per)}}\times\xpres{q}\bp{\R^n\times(0,\per)}
\end{align*}
is another solution. Then $0=\projcompl\hproj\bb{\ALTP\np{\ouvel-\uvel}\circ\bijectioninv}=\ALTP\bb{\projcompl\np{\ouvel-\uvel}\circ\bijectioninv}$. The injectivity of
$\ALTP$ established in Theorem \ref{lt_TPOseenMappingThm} thus implies $\projcompl\np{\ouvel-\uvel}\circ\bijectioninv=0$. 
Similarly, we see that $0=\proj\hproj\bb{\ALTP\np{\ouvel-\uvel}\circ\bijectioninv}=(\Delta-\rey\partial_1)\bb{\proj\np{\ouvel-\uvel}\circ\bijectioninv}$.
By well-known theory for the linearized, steady-state Navier-Stokes system, see again \cite[Theorem IV.2.1]{galdi:book1} and
\cite[Theorem VII.4.1]{galdi:book1}, $(\Delta-\rey\partial_1):\maxregspacelinnss{q}(\R^n)\ra\LRsigma{q}(\R^n)$ is a homeomorphism, whence 
$\proj\np{\ouvel-\uvel}\circ\bijectioninv=0$ follows. We can now conclude that $\ouvel-\uvel=0$. As a consequence hereof we then have $\grad\np{\oupres-\upres}=0$, from 
which, by the injectivity of $\gradmap$ established in Lemma \ref{lt_PressureMappingLem}, we deduce $\oupres-\upres=0$. 
Hence $(\ouvel,\oupres)=(\uvel,\upres)$.
\end{proof}

\bibliographystyle{abbrv}

\begin{thebibliography}{1}

\bibitem{Bruhat61}
F.~Bruhat.
\newblock {Distributions sur un groupe localement compact et applications \`a
  l'\'etude des repr\'esentations des groupes $p$-adiques.}
\newblock {\em Bull. Soc. Math. Fr.}, 89:43--75, 1961.

\bibitem{Leeuw1965}
K.~de~Leeuw.
\newblock {On $L\sb p$ multipliers.}
\newblock {\em Ann. Math. (2)}, 81:364--379, 1965.

\bibitem{EdwardsGaudryBook}
R.~Edwards and G.~Gaudry.
\newblock {\em {Littlewood-Paley and multiplier theory.}}
\newblock {Berlin-Heidelberg-New York: Springer-Verlag}, 1977.

\bibitem{galdi:book1}
G.~P. Galdi.
\newblock {\em {An introduction to the mathematical theory of the Navier-Stokes
  equations. Vol. I: Linearized steady problems}}.
\newblock {Springer Tracts in Natural Philosophy. 38. New York:
  Springer-Verlag}, 1994.

\bibitem{Grafakos1}
L.~Grafakos.
\newblock {\em {Classical Fourier analysis. 2nd ed.}}
\newblock {New York, NY: Springer}, 2008.

\bibitem{Grafakos2}
L.~Grafakos.
\newblock {\em {Modern Fourier analysis. 2nd ed.}}
\newblock {New York, NY: Springer}, 2009.
                                                                                                                                     
\bibitem{habil}                                                                                                                      
M.~Kyed.                                                                                                                             
\newblock {Time-Periodic Solutions to the Navier-Stokes Equations}.                                                                  
\newblock {\em {Habilitationsschrift, Technische Universit{\"a}t Darmstadt}},                                                        
  2012.                                                                                                                              
                                                                                                                                     
\bibitem{Stein70}                                                                                                                    
E.~M. Stein.                                                                                                                         
\newblock {\em {Singular integrals and differentiability properties of                                                               
  functions.}}                                                                                                                       
\newblock {Princeton, N.J.: Princeton University Press}, 1970.                                                                       
                                                                                                                                     
\end{thebibliography}

\end{document}